\documentclass[10pt]{amsart}
\usepackage{graphicx,amscd,amsmath,amsfonts,amssymb,geometry}
\usepackage{amssymb}
\usepackage{amsmath}
\usepackage{amsfonts}
\usepackage{graphicx}

\providecommand{\U}[1]{\protect\rule{.1in}{.1in}}

\newtheorem{theorem}{Theorem}[section]
\newtheorem{proposition}[theorem]{Proposition}

\newtheorem{remark}[theorem]{Remark}

\newtheorem{lemma}[theorem]{Lemma}

\numberwithin{equation}{section}

\begin{document}
\title[On the polynomial Hardy--Littlewood inequality]{On the polynomial
Hardy--Littlewood inequality}
\author[Ara\'{u}jo]{G. Ara\'{u}jo}
\address{Departamento de Matem\'{a}tica \\
\indent Universidade Federal da Para\'{\i}ba \\
\indent 58.051-900 - Jo\~{a}o Pessoa, Brazil.}
\email{gdasaraujo@gmail.com}
\author[Jim\'{e}nez]{P. Jim\'{e}nez-Rodriguez}
\address{Departament of Mathematical Sciences\\
\indent Mathematics and computer Sciences Building\\
\indent Kent State University \\
\indent Kent, OH 44242, USA.}
\email{pjimene1@kent.edu}
\author[Mu\~{n}oz]{G. A. Mu\~{n}oz-Fern\'{a}ndez}
\address{Departamento de An\'{a}lisis Matem\'{a}tico\\
\indent Facultad de Ciencias Matem\'{a}ticas \\
\indent Plaza de Ciencias 3 \\
\indent Universidad Complutense de Madrid\\
\indent Madrid, 28040, Spain.}
\email{gustavo$\_$fernandez@mat.ucm.es}
\author[N\'{u}\~{n}ez]{D. N\'{u}\~{n}ez-Alarc\'{o}n}
\address{Departamento de Matem\'{a}tica\\
\indent  Universidade Federal de Pernambuco\\
\indent 50.740-560 - Recife, Brazil.}
\email{danielnunezal@gmail.com}
\author[Pellegrino]{D. Pellegrino}
\address{Departamento de Matem\'{a}tica \\
\indent Universidade Federal da Para\'{\i}ba \\
\indent 58.051-900 - Jo\~{a}o Pessoa, Brazil.}
\email{pellegrino@pq.cnpq.br and dmpellegrino@gmail.com}
\author[Seoane]{J.B. Seoane-Sep\'{u}lveda}
\address{Departamento de An\'{a}lisis Matem\'{a}tico\\
\indent Facultad de Ciencias Matem\'{a}ticas \\
\indent Plaza de Ciencias 3 \\
\indent Universidad Complutense de Madrid\\
\indent Madrid, 28040, Spain.}
\email{jseoane@mat.ucm.es}
\author[Serrano]{D. M. Serrano-Rodr\'iguez}
\address{Departamento de Matem\'{a}tica\\
\indent  Universidade Federal de Pernambuco\\
\indent 50.740-560 - Recife, Brazil.}
\email{dmserrano0@gmail.com}
\thanks{G. Ara\'{u}jo, D. Pellegrino and J.B. Seoane-Sep\'{u}lveda was
supported by CNPq Grant 401735/2013-3 (PVE - Linha 2). G.A. Mu\~{n}oz-Fern%
\'{a}ndez and J. B. Seoane-Sep\'{u}lveda were supported by MTM2012-34341.}
\subjclass{}
\keywords{Hardy--Littlewood inequality, Bohnenblust--Hille inequality;
absolutely summing operators}

\begin{abstract}
We investigate the growth of the constants of the polynomial
Hardy-Littlewood inequality.
\end{abstract}

\maketitle

\section{Introduction}

Given $\alpha =(\alpha _{1},\ldots ,\alpha _{n})\in {\mathbb{N}}^{n}$, let $%
|\alpha |:=\alpha _{1}+\cdots +\alpha _{n}$ and let $\mathbf{x}^{\alpha }$
stand for the monomial $x_{1}^{\alpha _{1}}\cdots x_{n}^{\alpha _{n}}$ for $%
\mathbf{x}=(x_{1},\ldots ,x_{n})\in {\mathbb{K}}^{n}$. The polynomial
Bohnenblust--Hille inequality asserts that if $P$ is a homogeneous
polynomial of degree $m$ on $\ell _{\infty }^{n}$ given by 
\begin{equation*}
P(x_{1},...,x_{n})=\sum_{|\alpha |=m}a_{\alpha }\mathbf{{x}^{\alpha },}
\end{equation*}%
then 
\begin{equation*}
\left( {\sum\limits_{\left\vert \alpha \right\vert =m}}\left\vert a_{\alpha
}\right\vert ^{\frac{2m}{m+1}}\right) ^{\frac{m+1}{2m}}\leq D_{m}^{\mathbb{K}%
}\left\Vert P\right\Vert
\end{equation*}%
for some positive constant $D_{m}^{\mathbb{K}}$ which does not depend on $n$
(the exponent $\frac{2m}{m+1}$ is optimal). Precise estimates of the growth
of the constants $D_{m}^{\mathbb{K}}$ are crucial for different
applications. The following diagram shows the evolution of the estimates of $%
D_{m}^{\mathbb{C}}:$

\bigskip

\begin{center}
\begin{tabular}{|c|c|c|}
\hline
\textbf{Authors} & \textbf{Year} & \textbf{Estimate} \\ \hline
Bohnenblust and Hille & 1931, \cite{bh} & $D_{m}^{\mathbb{C}}\leq m^{\frac{%
m+1}{2m}}\left( \sqrt{2}\right) ^{m-1}$ \\ \hline
$%
\begin{array}{c}
\text{Defant, Frerick, Ortega-Cerd\'{a},} \\ 
\text{Ouna{\"{\i}}es, and Seip}%
\end{array}
$ & 2011, \cite{ann} & $D_{m}^{\mathbb{C}}\leq\left( 1+\frac{1}{m-1}\right)
^{m-1}\sqrt{m}\left( \sqrt{2}\right) ^{m-1} $ \\ \hline
Bayart, Pellegrino, and Seoane-Sep\'{u}lveda & 2013, \cite{bps} & $D_{m}^{%
\mathbb{C}}\leq C(\varepsilon)\left( 1+\varepsilon\right) ^{m}$, \\ \hline
\end{tabular}
\end{center}


\bigskip

\noindent where, in the table above, $C(\varepsilon)\left( 1+\varepsilon
\right) ^{m}$ means that given $\varepsilon>0$, there is a constant $C\left(
\varepsilon\right) >0$ such that $D_{m}^{\mathbb{C}}\leq C(\varepsilon
)\left( 1+\varepsilon\right) ^{m} $ for all $m.$

For real scalars it is shown in \cite{campos}\ that%
\begin{equation*}
\left( 1.1\right) ^{m}\leq D_{m}^{\mathbb{R}}\leq C(\varepsilon )\left(
2+\varepsilon \right) ^{m},
\end{equation*}%
and this means that for real scalars the inequality is hypercontractive and
this result is optimal.

When replacing $\ell _{\infty }^{n}$ by $\ell _{p}^{n}$ the extension of the
Bohnenblust--Hille inequality is called Hardy--Littlewood inequality and the
optimal exponents are $\frac{2mp}{mp+p-2m}$ for $p\geq 2m$. This is a
consequence of the multilinear Hardy--Littlewood inequality (see \cite%
{abps.hlw, dimant}). \ More precisely, if $P$ is a homogeneous polynomial of
degree $m$ on $\ell _{p}^{n},$ with $p\geq 2m,$ given by 
\begin{equation*}
P(x_{1},\ldots ,x_{n})=\sum_{|\alpha |=m}a_{\alpha }\mathbf{{x}^{\alpha },}
\end{equation*}%
then 
\begin{equation*}
\left( {\sum\limits_{\left\vert \alpha \right\vert =m}}\left\vert a_{\alpha
}\right\vert ^{\frac{2mp}{mp+p-2m}}\right) ^{\frac{mp+p-2m}{2mp}}\leq
D_{m,p}^{\mathbb{K}}\left\Vert P\right\Vert ,
\end{equation*}%
and $D_{m,p}^{\mathbb{K}}$ does not depend on $n$.

In this paper we look for upper and lower estimates for $D_{m,p}^{\mathbb{K}%
}.$

\section{First (and probably bad) upper estimates for $D_{m,p}^{\mathbb{K}}$}

Given $\alpha =(\alpha _{1},\ldots ,\alpha _{n})\in {\mathbb{N}}^{n}$, let
us define 
\begin{equation*}
\binom{m}{\alpha }:=\frac{m!}{\alpha _{1}!\cdot \ldots \cdot \alpha _{n}!},
\end{equation*}%
for $|\alpha |=m\in {\mathbb{N}}^{\ast }$. A straightforward consequence of
the multinomial formula yields the following relationship between the
coefficients of a homogeneous polynomial and the polar of the polynomial
(this lemma appears in \cite{diaz} and is essentially \emph{folklore}).

\begin{lemma}
\label{lem:Multinomial} If $P$ is a homogeneous polynomial of degree $m$ on $%
{\mathbb{K}}^{n}$ given by 
\begin{equation*}
P(x_{1},\ldots,x_{n})=\sum_{|\alpha|=m}a_{\alpha}\mathbf{{x}^{\alpha},}
\end{equation*}
and $L$ is the polar of $P$, then 
\begin{equation*}
L(e_{1}^{\alpha_{1}},\ldots,e_{n}^{\alpha_{n}})=\frac{a_{\alpha}}{\binom {m}{%
\alpha}},
\end{equation*}
where $\{e_{1},\ldots,e_{n}\}$ is the canonical basis of ${\mathbb{K}}^{n}$
and $e_{k}^{\alpha_{k}}$ stands for $e_{k}$ repeated $\alpha_{k}$ times.
\end{lemma}

\bigskip From now on, for any map $f:\mathbb{R}\rightarrow \mathbb{R}$ we
define 
\begin{equation*}
f\left( \infty \right) :=\lim_{p\rightarrow \infty }f(p).
\end{equation*}

The following result is also essentially known. We present here the details
of its proof for the sake of completeness of the paper.

\begin{proposition}
\label{pro:first_approach} If $P$ is a homogeneous polynomial of degree $m$
on $\ell _{p}^{n},$ with $p\geq 2m,$ given by 
\begin{equation*}
P(x_{1},\ldots ,x_{n})=\sum_{|\alpha |=m}a_{\alpha }\mathbf{{x}^{\alpha },}
\end{equation*}%
then 
\begin{equation*}
\left( {\sum\limits_{\left\vert \alpha \right\vert =m}}\left\vert a_{\alpha
}\right\vert ^{\frac{2mp}{mp+p-2m}}\right) ^{\frac{mp+p-2m}{2mp}}\leq
D_{m,p}^{\mathbb{K}}\left\Vert P\right\Vert 
\end{equation*}%
with 
\begin{equation*}
D_{m,p}^{\mathbb{K}}=C_{m,p}^{\mathbb{K}}\frac{m^{m}}{\left( m!\right) ^{%
\frac{mp+p-2m}{2mp}}},
\end{equation*}%
where $C_{m,p}^{\mathbb{K}}$ are the constants of the multilinear
Hardy-Littlewood inequality.
\end{proposition}

\begin{proof}
From Lemma \ref{lem:Multinomial} we have 
\begin{align*}
{\sum\limits_{\left\vert \alpha\right\vert =m}}\left\vert a_{\alpha
}\right\vert ^{\frac{2mp}{mp+p-2m}} & ={\sum\limits_{\left\vert
\alpha\right\vert =m}}\left( \binom{m}{\alpha}\left\vert
L(e_{1}^{\alpha_{1}},\ldots,e_{n}^{\alpha_{n}})\right\vert \right) ^{\frac{%
2mp}{mp+p-2m}} \\
& ={\sum\limits_{\left\vert \alpha\right\vert =m}}\binom{m}{\alpha}^{\frac{%
2mp}{mp+p-2m}}\left\vert L(e_{1}^{\alpha_{1}},\ldots,e_{n}^{\alpha
_{n}})\right\vert ^{\frac{2mp}{mp+p-2m}}.
\end{align*}
However, for every choice of $\alpha$, the term 
\begin{equation*}
\left\vert L(e_{1}^{\alpha_{1}},\ldots,e_{n}^{\alpha_{n}})\right\vert ^{%
\frac{2mp}{mp+p-2m}}
\end{equation*}
is repeated $\binom{m}{\alpha}$ times in the sum 
\begin{equation*}
{\sum\limits_{i_{1},\ldots,i_{m}=1}^{n}}\left\vert L(e_{i_{1}},\ldots
,e_{i_{m}})\right\vert ^{\frac{2mp}{mp+p-2m}}.
\end{equation*}
Thus 
\begin{equation*}
{\sum\limits_{\left\vert \alpha\right\vert =m}}\binom{m}{\alpha}^{\frac {2mp%
}{mp+p-2m}}\left\vert
L(e_{1}^{\alpha_{1}},\ldots,e_{n}^{\alpha_{n}})\right\vert ^{\frac{2mp}{%
mp+p-2m}}={\sum\limits_{i_{1},\ldots,i_{m}=1}^{n}}\binom{m}{\alpha}^{\frac{%
2mp}{mp+p-2m}}\frac{1}{\binom{m}{\alpha}}\left\vert
L(e_{i_{1}},\ldots,e_{i_{m}})\right\vert ^{\frac{2mp}{mp+p-2m}}
\end{equation*}
and, since 
\begin{equation*}
\binom{m}{\alpha}\leq m!
\end{equation*}
we have 
\begin{equation*}
{\sum\limits_{\left\vert \alpha\right\vert =m}}\binom{m}{\alpha}^{\frac {2mp%
}{mp+p-2m}}\left\vert
L(e_{1}^{\alpha_{1}},\ldots,e_{n}^{\alpha_{n}})\right\vert ^{\frac{2mp}{%
mp+p-2m}}\leq\left( m!\right) ^{\frac {mp-p+2m}{mp+p-2m}}{%
\sum\limits_{i_{1},\ldots,i_{m}=1}^{n}}\left\vert
L(e_{i_{1}},\ldots,e_{i_{m}})\right\vert ^{\frac{2mp}{mp+p-2m}}.
\end{equation*}
We finally obtain 
\begin{align*}
\left( {\sum\limits_{\left\vert \alpha\right\vert =m}}\left\vert a_{\alpha
}\right\vert ^{\frac{2mp}{mp+p-2m}}\right) ^{\frac{mp+p-2m}{2mp}} &
\leq\left( \left( m!\right) ^{\frac{mp-p+2m}{mp+p-2m}}{\sum\limits_{i_{1},%
\ldots,i_{m}=1}^{n}}\left\vert L(e_{i_{1}},\ldots,e_{i_{m}})\right\vert ^{%
\frac{2mp}{mp+p-2m}}\right) ^{\frac{mp+p-2m}{2mp}} \\
& =\left( m!\right) ^{\frac{mp-p+2m}{2mp}}\left( {\sum\limits_{i_{1},%
\ldots,i_{m}=1}^{n}}\left\vert L(e_{i_{1}},\ldots,e_{i_{m}})\right\vert ^{%
\frac{2mp}{mp+p-2m}}\right) ^{\frac{mp+p-2m}{2mp}} \\
& \leq\left( m!\right) ^{\frac{mp-p+2m}{2mp}}C_{m,p}^{\mathbb{K}}\left\Vert
L\right\Vert .
\end{align*}
On the other hand, it is well-known that 
\begin{equation*}
\left\Vert L\right\Vert \leq\frac{m^{m}}{m!}\left\Vert P\right\Vert
\end{equation*}
and, hence, 
\begin{align*}
\left( {\sum\limits_{\left\vert \alpha\right\vert =m}}\left\vert a_{\alpha
}\right\vert ^{\frac{2mp}{mp+p-2m}}\right) ^{\frac{mp+p-2m}{2mp}} & \leq
C_{m,p}^{\mathbb{K}}\left( m!\right) ^{\frac{mp-p+2m}{2mp}}\frac{m^{m}}{m!}%
\left\Vert P\right\Vert \\
& =C_{m,p}^{\mathbb{K}}\frac{m^{m}}{\left( m!\right) ^{\frac{mp+p-2m}{2mp}}}%
\left\Vert P\right\Vert .
\end{align*}
\end{proof}

\begin{remark}
Let us define the polarization constants for polynomials on $\ell_{p}$
spaces as 
\begin{equation*}
{\mathbb{K}}(m,p):=\inf\{M> 0\, :\, \|L\|\leq M\|P\|\},
\end{equation*}
where the infimum is taken over all $P\in{\mathcal{P}}(^{m}\ell_{p}^{n})$
and $L$ is the polar of $P$. Notice that ${\mathbb{K}}(m,p)$ may improve the
inequality 
\begin{equation*}
\left\Vert L\right\Vert \leq\frac{m^{m}}{m!}\left\Vert P\right\Vert
\end{equation*}
used in the proof of Proposition \ref{pro:first_approach}. Using ${\mathbb{K}%
}(m,p)$ instead of $\frac{m^{m}}{m!}$ in the proof of Proposition \ref%
{pro:first_approach} we end up with 
\begin{equation*}
D_{m,p}^{\mathbb{K}}=\left( m!\right) ^{\frac{mp-p+2m}{2mp}}C_{m,p}^{\mathbb{%
K}}{\mathbb{K}}(m,p).
\end{equation*}

The calculation of polarization constants has been a fruitful problem since
the 1970's and some facts are know about them. Specifically, Harris proved
in \cite{Harris} that 
\begin{equation*}
{\mathbb{C}}(m,p)\leq \left(\frac{m^m}{m!}\right)^\frac{|p-2|}{p},
\end{equation*}
for all $p\geq 1$, whenever $m$ is a power of $2$. Also, Harris \cite{Harris}
proved that 
\begin{equation*}
{\mathbb{C}}(m,\infty)\leq \frac{m^\frac{m}{2}(m+1)^\frac{m+1}{2}}{2^mm!}.
\end{equation*}
One more example was provided by Sarantopoulos \cite{Sarant} who proved that 
\begin{equation*}
{\mathbb{K}}(m,p)=\frac{m^\frac{m}{p}}{m!},
\end{equation*}
whenever $1\leq p\leq \frac{m}{m-1}$.
\end{remark}

\section{The real polynomial Hardy--Littlewood inequality: lower bounds for
the constants}

For $m\geq 2$ and $p\geq 2m$, let us denote by $H_{\mathbb{K},m,p}^{\mathrm{%
pol}}$ the optimal constants satisfying the polynomial Hardy-Littlewood
inequality with scalars in $\mathbb{K}.$ In this section we show that for
real scalars the polynomial Hardy-Littlewood inequality has at least an
hypercontractive growth.

\begin{theorem}
\label{88990}For all positive integers $m\geq 2$ and $2m\leq p<\infty $ we
have%
\begin{equation*}
\left( \sqrt[16]{2}\right) ^{m}\leq 2^{\frac{mp+p-6m+4}{4p}\cdot \frac{m-1}{m%
}}\leq H_{\mathbb{R},m,p}^{\mathrm{pol}}.
\end{equation*}
\end{theorem}

\bigskip

\begin{proof}
Let $m$ be an even integer. Consider the $m$-homogeneous polynomial $%
P_{m}:\ell _{p}^m\rightarrow \mathbb{R}$ given by 
\begin{equation*}
P_{m}(x_{1},\ldots ,x_{m})=\left( x_{1}^{2}-x_{2}^{2}\right) \left(
x_{3}^{2}-x_{4}^{2}\right) \cdots \left( x_{m-1}^{2}-x_{m}^{2}\right) .
\end{equation*}

Notice that 
\begin{equation*}
\left\Vert P_{m}\right\Vert =P_{m}\left( \frac{1}{\sqrt[p]{m/2}},0,\frac{1}{%
\sqrt[p]{m/2}},...,\frac{1}{\sqrt[p]{m/2}},0\right) =\left( \frac{1}{\sqrt[p]%
{m/2}}\right) ^{m}.
\end{equation*}%
From the Hardy-Littlewood inequality for $P_{m}$ we have 
\begin{equation*}
\left( {\sum\limits_{\left\vert \alpha \right\vert =m}}\left\vert a_{\alpha
}\right\vert ^{\frac{2mp}{mp+p-2m}}\right) ^{\frac{mp+p-2m}{2mp}}\leq H_{%
\mathbb{R},m,p}^{\mathrm{pol}}\left\Vert P_{m}\right\Vert ,
\end{equation*}%
i.e., 
\begin{equation*}
H_{\mathbb{R},m,p}^{\mathrm{pol}}\geq \frac{\left( 2^{\frac{m}{2}}\right) ^{%
\frac{mp+p-2m}{2mp}}}{\left( \frac{1}{\sqrt[p]{m/2}}\right) ^{m}}=2^{\frac{%
mp+p-2m}{4p}}\left( \frac{m}{2}\right) ^{\frac{m}{p}}=2^{\frac{mp+p-6m}{4p}%
}m^{\frac{m}{p}}\geq 2^{\frac{mp+p-6m}{4p}}.
\end{equation*}

If $m$ is odd, define%
\begin{equation*}
Q_{m}(x_{1},\ldots,x_{m})=\left( x_{1}^{2}-x_{2}^{2}\right) \left(
x_{3}^{2}-x_{4}^{2}\right) \cdots \left( x_{m-2}^{2}-x_{m-1}^{2}\right)
x_{m}.
\end{equation*}
Then%
\begin{equation*}
\left\Vert Q_{m}\right\Vert \leq\left\Vert P_{m-1}\right\Vert =\left( \frac{1%
}{\sqrt[p]{\left( m-1\right) /2}}\right) ^{m-1}.
\end{equation*}
From the Hardy-Littlewood inequality for $Q_{m}$ we have 
\begin{equation*}
\left( {\sum\limits_{\left\vert \alpha\right\vert =m}}\left\vert a_{\alpha
}\right\vert ^{\frac{2mp}{mp+p-2m}}\right) ^{\frac{mp+p-2m}{2mp}}\leq H_{%
\mathbb{R},m,p}^{\mathrm{pol}}\left\Vert Q_{m}\right\Vert ,
\end{equation*}
i.e., 
\begin{eqnarray*}
H_{\mathbb{R},m,p}^{\mathrm{pol}} & \geq & \frac{\left( 2^{\frac{m-1}{2}%
}\right) ^{\frac{mp+p-2m}{2mp}}}{\left( \frac{1}{\sqrt[p]{\left( m-1\right)
/2}}\right) ^{m-1}} \\
& = & 2^{\frac{\left( mp+p-2m\right) \left( m-1\right) }{4mp}}\left( \frac{%
m-1}{2}\right) ^{\frac{m-1}{p}} \\
& = & 2^{\frac{mp+p-6m+4}{4p}\cdot \frac{m-1}{m} }\left( m-1\right) ^{\frac{%
m-1}{p}} \\
& \geq & 2^{\frac{mp+p-6m+4}{4p}\cdot\frac{m-1}{m}}.
\end{eqnarray*}
\end{proof}

\begin{remark}
From the estimates of the last proof note that if $m$ is even, and $m\geq 4$%
, then 
\begin{equation}
\begin{array}{rcl}
H_{\mathbb{R},m,p}^{\mathrm{pol}} & \geq  & 2^{\frac{mp+p-6m}{4p}}m^{\frac{m%
}{p}}\vspace{0.2cm} \\ 
& \geq  & 2^{\frac{mp+p-6m}{4p}}\left( 2^{\frac{3}{2}}\right) ^{\frac{m}{p}}%
\vspace{0.2cm} \\ 
& = & \left( \sqrt[4]{2}\right) ^{m+1}.%
\end{array}
\label{par}
\end{equation}%
If $m$ is odd, and $m\geq 5$, then 
\begin{equation}
\begin{array}{rcl}
H_{\mathbb{R},m,p}^{\mathrm{pol}} & \geq  & 2^{\frac{mp+p-6m+4}{4p}\cdot 
\frac{m-1}{m}}\left( m-1\right) ^{\frac{m-1}{p}}\vspace{0.2cm} \\ 
& \geq  & 2^{\frac{mp+p-6m+4}{4p}\cdot \frac{m-1}{m}}\left( 2^{\frac{3}{2}%
}\right) ^{\frac{m-1}{p}}\vspace{0.2cm} \\ 
& = & \allowbreak 2^{\frac{pm^{2}+4m-p-4}{4mp}}\vspace{0.2cm} \\ 
& \geq  & \left( \sqrt[4]{2}\right) ^{m-\frac{1}{m}}.%
\end{array}
\label{impar}
\end{equation}%
Thus, by (\ref{par}) and (\ref{impar}), if $m\geq 4$ 
\begin{equation*}
H_{\mathbb{R},m,p}^{\mathrm{pol}}\geq \left( \sqrt[4]{2}\right) ^{m-\frac{1}{%
m}}.
\end{equation*}
\end{remark}

\section{Real versus complex estimates}

As it happens with the constants of the Bohnenblust-Hille inequality, we
observe that%
\begin{equation}
H_{\mathbb{R},m,p}^{\mathrm{pol}}\leq 2^{m-1}H_{\mathbb{C},m,p}^{\mathrm{pol}%
}.  \label{989}
\end{equation}%
In fact, from \cite{muno} we know that if $P:\ell _{p}\rightarrow \mathbb{R}$
is an $m$-homogeneous polynomial and $P_{\mathbb{C}}:\ell _{p}\rightarrow 
\mathbb{C}$ is the same polynomial, then%
\begin{equation*}
\left\Vert P_{\mathbb{C}}\right\Vert \leq 2^{m-1}\left\Vert P\right\Vert .
\end{equation*}%
We thus obtain (\ref{989}). So if one succeeds in proving that the constants
of the complex Hardy-Littlewood polynomial inequality have a subexponential
growth (as it happens with the constants of the complex polynomial
Bohnenblust-Hille inequality) then we immediately conclude that the
constants of the real polynomial Hardy littlewood inequality have an
exponential growth and this result is optimal (due to Theorem \ref{88990}).

\section{The complex polynomial Hardy-Littlewood inequality: trying to find
upper estimates}

\bigskip In this section we try to improve the estimates for the upper
bounds of the complex polynomial Hardy-Littlewwod inequality from Section 2.
However, for $p<\infty $ the validity of our estimates depend on a
conjecture (see (\ref{888})). For the case $p=\infty $ our results are
exactly those from (\cite[Remark 2.2]{bps}). 

The following multi-index notation will come in handy for us: for positive
integers $m,n$, we set 
\begin{align*}
\mathcal{M}(m,n)& :=\left\{ \mathbf{i}=(i_{1},\dots ,i_{m});\,i_{1},\dots
,i_{m}\in \{1,\dots ,n\}\right\} , \\
\mathcal{J}(m,n)& :=\left\{ \mathbf{i}\in \mathcal{M}(m,n);\,i_{1}\leq
i_{2}\leq \dots \leq i_{m}\right\} ,
\end{align*}%
and for $k=1,\dots ,m$, $\mathcal{P}_{k}(m)$ denotes the set of the subsets
of $\{1,\dots ,m\}$ with cardinality $k$. For $S=\{s_{1},\dots ,s_{k}\}\in 
\mathcal{P}_{k}(m)$, its complement will be $\widehat{S}:=\{1,\dots
,m\}\setminus S$, and $\mathbf{i}_{S}$ shall mean $(i_{s_{1}},\dots
,i_{s_{k}})\in \mathcal{M}(k,n)$. For a multi-index $\mathbf{i}\in \mathcal{M%
}(m,n)$, we denote by $|\mathbf{i}|$ the cardinality of the set of
multi-indexes $\mathbf{j}\in \mathcal{M}(m,n)$ such that there is a
permutation $\sigma $ of $\{1,\dots ,m\}$ with $i_{\sigma (k)}=j_{k}$, for
every $k=1,\dots ,m$. The equivalence class of $\mathbf{i}$ is denoted by $[%
\mathbf{i].}$ When we write $c_{[\mathbf{i]}}$ for $\mathbf{i\in }\mathcal{M}%
(m,n)$ we mean $c_{\mathbf{j}}$ for $\mathbf{j}\in \mathcal{J}(m,n)$ and $%
\mathbf{j}$ equivalent to $\mathbf{i}.$

The following very recent generalization of the famous Blei inequality will
be crucial for our estimates (see \cite[Remark 2.2]{bps}).

\begin{lemma}[Bayart, Pellegrino, Seoane, \protect\cite{bps}]
\label{blei.interp} Let $m,n$ positive integers, $1\leq k\leq m$ and $1\leq
s\leq q$, satisfying $\frac{m}{\rho }=\frac{k}{s}+\frac{m-k}{q}$. Then for
all scalar matrix $\left( a_{\mathbf{i}}\right) _{\mathbf{i}\in\mathcal{M}%
(m,n)}$,

\begin{equation*}
\left( \sum_{\mathbf{i}\in\mathcal{M}(m,n)}\left\vert a_{\mathbf{i}%
}\right\vert ^{\rho}\right) ^{\frac{1}{\rho}}\leq\prod_{S\in\mathcal{P}%
_{k}(m)}\left( \sum_{\mathbf{i}_{S}}\left( \sum_{\mathbf{i}_{\widehat{S}%
}}\left\vert a_{\mathbf{i}}\right\vert ^{q}\right) ^{\frac{s}{q}}\right) ^{%
\frac{1}{s}\cdot\frac{1}{\binom{m}{k}}}.
\end{equation*}
\end{lemma}

Let us use the following notation: $S_{\ell _{p}^{n}}$ denotes the unit
sphere on $\ell _{p}^{n}$ if $p<\infty $, and $S_{\ell _{\infty }^{n}}$
denotes the $n$-dimensional torus. More precisely: for $p\in \left( 0,\infty
\right) $%
\begin{equation*}
S_{\ell _{p}^{n}}:=\left\{ \mathbf{z}=\left( z_{1},...,z_{n}\right) \in 
\mathbb{C}^{n}:\left\Vert \mathbf{z}\right\Vert _{\ell _{p}^{n}}=1\right\} ,
\end{equation*}%
and%
\begin{equation*}
S_{\ell _{\infty }^{n}}:=\mathbb{T}^{n}=\left\{ \mathbf{z}=\left(
z_{1},...,z_{n}\right) \in \mathbb{C}^{n}:\left\vert z_{i}\right\vert
=1\right\} .
\end{equation*}
For $p\in \left[ 1,\infty \right] $, let $\mu _{p}^{n}$ be the normalized
Lebesgue measure on $S_{\ell _{p}^{n}}$. For the sake of simplicity, $\mu
_{p}^{n}$ will be simply denoted by $\mu ^{n}.$ The following inequality is
a variant of a result due to Bayart \cite[Theorem 9]{bayart} (see, for
instance, \cite[Lemma 5.1]{bps}).

\begin{lemma}[Bayart]
\label{bay_in} Let $1\leq s\leq 2$. For every $m$-homogeneous polynomial $P(%
\mathbf{z})=\sum_{|\alpha |=m}a_{\alpha }\mathbf{z}^{\alpha }$ on $\mathbb{C}%
^{n}$ with values in $\mathbb{C}$, we have 
\begin{equation*}
\left( \sum_{|\alpha |=m}\left\vert a_{\alpha }\right\vert ^{2}\right) ^{%
\frac{1}{2}}\leq \left( \frac{2}{s}\right) ^{\frac{m}{2}}\left( \int_{%
\mathbb{T}^{n}}\left\vert P(\mathbf{z})\right\vert ^{s}d\mu ^{n}(\mathbf{z}%
)\right) ^{\frac{1}{s}}.
\end{equation*}
\end{lemma}

For $m\in \lbrack 2,\infty ]$ let us define $p_{0}(m)$ as the infimum of the
values of $p\in \left[ 2m,\infty \right] $ such that for all $1\leq s\leq 
\frac{2p}{p-2}$ there is a $K_{s,p}>0$ such that%
\begin{equation}
\left( \sum_{|\alpha |=m}\left\vert a_{\alpha }\right\vert ^{\frac{2p}{p-2}%
}\right) ^{\frac{p-2}{2p}}\leq K_{s,p}^{m}\left( \int_{S_{\ell
_{p}^{n}}}\left\vert P(\mathbf{z})\right\vert ^{s}d\mu ^{n}(\mathbf{z}%
)\right) ^{\frac{1}{s}}  \label{888}
\end{equation}%
for all positive integers $n$ and all $m$-homogeneous polynomials $P:\mathbb{%
C}^{n}\rightarrow \mathbb{C}$. For the sake of simplicity $p_{0}(m)$ will be
simply denoted by $p_{0}.$ From Bayart's lemma we know that this definition
makes sense, since from Bayart's lemma we know that (\ref{888}) is valid for 
$p=\infty .$ We conjecture that $p_{0}\leq m^{2}.$


Now, let us state and prove the main result of this section. The argument of
the proof follows the lines of that in \cite{ann, bps}. We will use the
following result due L. Harris (see \cite[Theorem 1]{Harris}):

\begin{lemma}[Harris]
\label{haharr} Let $X$ be a complex normed linear space. If $P$ is a
homogeneous polynomial of degree $m$ on $X$ and $L$ is the polar of $P$,
then, for any nonnegative integers $m_1,...,m_k$ with $m_1+\cdots+m_k=m$ and
for any $x^{(1)},...,x^{(k)}$ unit vectors in $X$, 
\begin{equation*}
|L( \underbrace{x^{(1)},...,x^{(1)}}_{m_1 \text{ times}},...,\underbrace{%
x^{(k)},...,x^{(k)}}_{m_k \text{ times}})|\leq \frac{m_1 ! \cdots m_k !\cdot
m^m}{m_1^{m_1}\cdots m_k^{m_k}\cdot m!}\|P\|.
\end{equation*}
\end{lemma}

\begin{theorem}
\label{hardy}Let $m\in \lbrack 2,\infty ]$ and $1\leq k\leq m-1.$ If $%
p_{0}(m-k)<p\leq \infty $ (and $p=\infty $ if $p_{0}(m-k)=\infty $) then,
for every $m$-homogeneous polynomial $P:\ell _{p}^{n}\rightarrow \mathbb{C}$%
, defined by $P(\mathbf{z})=\sum_{|\alpha |=m}a_{\alpha }\mathbf{z}^{\alpha
},$ we have 
\begin{align*}
& \left( \sum_{|\alpha |=m}\left\vert a_{\alpha }\right\vert ^{\frac{2mp}{%
mp+p-2m}}\right) ^{\frac{mp+p-2m}{2mp}} \\
& \leq K_{\frac{2kp}{kp+p-2k},p}^{m-k}\cdot \frac{m^{m}}{(m-k)^{m-k}}\cdot
\left( \frac{(m-k)!}{m!}\right) ^{\frac{p-2}{2p}}\left( \frac{2}{\sqrt{\pi }}%
\right) ^{\frac{2k\left( k-1\right) }{p}}\cdot \left( B_{\mathbb{C},k}^{%
\mathrm{mult}}\right) ^{\frac{p-2k}{p}}\Vert P\Vert ,
\end{align*}%
where $B_{\mathbb{C},k}^{\mathrm{mult}}$ is any constant satisfying the $k$%
-linear Bohnenblust-Hille inequality.
\end{theorem}

\begin{proof}
We can also write 
\begin{equation*}
P\left( \mathbf{z}\right) =\sum_{\mathbf{i\in}\mathcal{J}(m,n)}c_{\mathbf{i}%
}z_{i_{1}}....z_{i_{m}}.
\end{equation*}
Consider 
\begin{align*}
\rho & =\frac{2mp}{mp+p-2m}, \\
s_{k} & =\frac{2kp}{kp+p-2k}, \\
q & =\frac{2p}{p-2}.
\end{align*}
Note that%
\begin{equation*}
s_{k}\leq2<q,
\end{equation*}%
\begin{equation*}
\frac{m}{\rho}=\frac{mp+p-2m}{2p},
\end{equation*}
and 
\begin{align*}
\frac{k}{s_{k}}+\frac{m-k}{q} & =\frac{kp+p-2k}{2p}+\frac{\left( m-k\right)
\left( p-2\right) }{2p} \\
& =\frac{kp+p-2k}{2p}+\frac{mp-kp-2m+2k}{2p} \\
& =\frac{mp+p-2m}{2p}.
\end{align*}
and thus%
\begin{equation*}
\frac{m}{\rho}=\frac{k}{s_{k}}+\frac{m-k}{q}
\end{equation*}
and we can use Lemma \ref{blei.interp}.

Let $L:\ell_{p}^{n}\times\dots\times\ell_{p}^{n}\rightarrow\mathbb{C}$ be
the unique symmetric $m$-linear map associated to $P$. Note that%
\begin{equation*}
L\left( \mathbf{z}^{(1)},...,\mathbf{z}^{(m)}\right) =\sum_{\mathbf{i\in}%
\mathcal{M}(m,n)}\frac{c_{[\mathbf{i}]}}{\left\vert \mathbf{i}\right\vert }%
z_{i_{1}}^{(1)}...z_{i_{m}}^{(m)}.
\end{equation*}
Thus%
\begin{align*}
\sum_{|\alpha|=m}\left\vert a_{\alpha}\right\vert ^{\frac{2mp}{mp+p-2m}} &
=\sum_{\mathbf{i\in}\mathcal{J}(m,n)}\left\vert c_{\mathbf{i}}\right\vert ^{%
\frac{2mp}{mp+p-2m}} \\
& =\sum_{\mathbf{i\in}\mathcal{M}(m,n)}\left\vert \mathbf{i}\right\vert ^{%
\frac{-p}{mp+p-2m}}\left( \frac{\left\vert c_{[\mathbf{i}]}\right\vert }{%
\left\vert \mathbf{i}\right\vert ^{\frac{1}{q}}}\right) ^{\frac {2mp}{mp+p-2m%
}} \\
& \leq\sum_{\mathbf{i\in}\mathcal{M}(m,n)}\left( \frac{\left\vert c_{\left[ 
\mathbf{i}\right] }\right\vert }{\left\vert \mathbf{i}\right\vert ^{\frac {1%
}{q}}}\right) ^{\frac{2mp}{mp+p-2m}}
\end{align*}
Using Lemma \ref{blei.interp} with $s_{k}=\frac{2kp}{kp+p-2k}$ and $q=\frac{%
2p}{p-2}$, we get 
\begin{align*}
& \left( \sum_{|\alpha|=m}\left\vert a_{\alpha}\right\vert ^{\frac {2mp}{%
mp+p-2m}}\right) ^{\frac{mp+p-2m}{2mp}} \\
& \leq\left[ \prod_{S\in\mathcal{P}_{k}}\left( \sum_{\mathbf{i}_{S}\in%
\mathcal{M}(k,n)}\left( \sum_{\mathbf{i}_{\hat{S}}\in\mathcal{M}%
(m-k,n)}\left( \frac{\left\vert c_{\left[ \mathbf{i}\right] }\right\vert }{%
\left\vert \mathbf{i}\right\vert ^{\frac{1}{q}}}\right) ^{q}\right) ^{\frac{%
s_{k}}{q}}\right) ^{\frac{1}{s_{k}}}\right] ^{\frac{1}{\binom{m}{k}}}.
\end{align*}

Note that $|\mathbf{i}|\leq |\mathbf{i}_{\hat{S}}|\frac{m!}{(m-k)!}$, and
thus 
\begin{align*}
& \left( \sum_{|\alpha |=m}\left\vert a_{\alpha }\right\vert ^{\frac{2mp}{%
mp+p-2m}}\right) ^{\frac{mp+p-2m}{2mp}} \\
& \leq \left( \frac{m!}{(m-k)!}\right) ^{\frac{q-1}{q}}\left[ \prod_{S\in 
\mathcal{P}_{k}}\left( \sum_{\mathbf{i}_{S}\in \mathcal{M}(k,n)}\left( \sum_{%
\mathbf{i}_{\hat{S}}\in \mathcal{M}(m-k,n)}\frac{\left\vert c_{\left[ 
\mathbf{i}\right] }\right\vert ^{q}}{\left\vert \mathbf{i}\right\vert }%
\left( \frac{\left\vert \mathbf{i}_{\hat{S}}\right\vert }{\left\vert \mathbf{%
i}\right\vert }\right) ^{q-1}\right) ^{\frac{s_{k}}{q}}\right) ^{\frac{1}{%
s_{k}}}\right] ^{\frac{1}{\binom{m}{k}}} \\
& =\left( \frac{m!}{(m-k)!}\right) ^{\frac{q-1}{q}}\left[ \prod_{S\in 
\mathcal{P}_{k}}\left( \sum_{\mathbf{i}_{S}\in \mathcal{M}(k,n)}\left( \sum_{%
\mathbf{i}_{\hat{S}}\in \mathcal{M}(m-k,n)}\frac{\left\vert c_{\left[ 
\mathbf{i}\right] }\right\vert ^{q}}{\left\vert \mathbf{i}\right\vert ^{q}}%
\left\vert \mathbf{i}_{\hat{S}}\right\vert ^{q-1}\right) ^{\frac{s_{k}}{q}%
}\right) ^{\frac{1}{s_{k}}}\right] ^{\frac{1}{\binom{m}{k}}}.
\end{align*}%
Let us fix $S\in \mathcal{P}_{k}(m)$. There is no loss of generality in
supposing $S=\left\{ 1,...,k\right\} $. We then fix some $\mathbf{i}_{S}\in 
\mathcal{M}(k,n)$ and we introduce the following $\left( m-k\right) $%
-homogeneous polynomial on $\ell _{p}^{n}$:%
\begin{equation*}
P_{\mathbf{i}_{S}}(\mathbf{z})=L\left( e_{i_{1}},...,e_{i_{k}},\mathbf{z}%
,...,\mathbf{z}\right) .
\end{equation*}%
Observe that 
\begin{align*}
P_{\mathbf{i}_{S}}(\mathbf{z})& =\sum_{\mathbf{i}_{\hat{S}}\in \mathcal{M}%
(m-k,n)}\frac{c_{\left[ \mathbf{i}\right] }}{\left\vert \mathbf{i}%
\right\vert }z_{\mathbf{i}_{\hat{S}}} \\
& =\sum_{\mathbf{i}_{\hat{S}}\in \mathcal{J}(m-k,n)}\frac{c_{\left[ \mathbf{i%
}\right] }}{\left\vert \mathbf{i}\right\vert }\left\vert \mathbf{i}_{\hat{S}%
}\right\vert z_{\mathbf{i}_{\hat{S}}}
\end{align*}%
so that 
\begin{equation*}
\left\Vert P_{\mathbf{i}_{S}}(\mathbf{z})\right\Vert _{q}=\left( \sum_{%
\mathbf{i}_{\hat{S}}\in \mathcal{J}(m-k,n)}\frac{\left\vert c_{\left[ 
\mathbf{i}\right] }\right\vert ^{q}}{\left\vert \mathbf{i}\right\vert ^{q}}%
\left\vert \mathbf{i}_{\hat{S}}\right\vert ^{q}\right) ^{\frac{1}{q}}=\left(
\sum_{\mathbf{i}_{\hat{S}}\in \mathcal{M}(m-k,n)}\frac{\left\vert c_{\left[ 
\mathbf{i}\right] }\right\vert ^{q}}{\left\vert \mathbf{i}\right\vert ^{q}}%
\left\vert \mathbf{i}_{\hat{S}}\right\vert ^{q-1}\right) ^{\frac{1}{q}}.
\end{equation*}%
By the definition of $p_{0}$ we have%
\begin{equation*}
\left\Vert P_{\mathbf{i}_{S}}(\mathbf{z})\right\Vert _{q}^{s_{k}}\leq
K_{s_{k},p}^{(m-k)s_{k}}\int_{S_{\ell _{p}^{n}}}\left\vert L\left(
e_{i_{1}},...,e_{i_{k}},\mathbf{z},...,\mathbf{z}\right) \right\vert
^{s_{k}}d\mu ^{n}(\mathbf{z}).
\end{equation*}%
Thus, 
\begin{align*}
& \sum_{\mathbf{i}_{S}}\left( \sum_{\mathbf{i}_{\hat{S}}}\frac{\left\vert c_{%
\left[ \mathbf{i}\right] }\right\vert ^{q}}{\left\vert \mathbf{i}\right\vert
^{q}}\left\vert \mathbf{i}_{\hat{S}}\right\vert ^{q-1}\right) ^{\frac{1}{q}%
\times s_{k}} \\
& \leq K_{s_{k},p}^{(m-k)s_{k}}\int_{S_{\ell _{p}^{n}}}\sum_{\mathbf{i}%
_{S}}\left\vert L\left( e_{i_{1}},...,e_{i_{k}},\mathbf{z},...,\mathbf{z}%
\right) \right\vert ^{s_{k}}d\mu ^{n}(\mathbf{z}).
\end{align*}%
Now fixing $\mathbf{z}\in S_{\ell _{p}^{n}}$ we apply the multilinear
Hardy-Littlewood inequality to the $k-$linear form%
\begin{equation*}
\left( \mathbf{z}^{(1)},...,\mathbf{z}^{(k)}\right) \mapsto L\left( \mathbf{z%
}^{(1)},...,\mathbf{z}^{(k)},\mathbf{z},...,\mathbf{z}\right)
\end{equation*}%
and we obtain, from \cite[Theorem 1.1]{apd} and Lemma \ref{haharr}, 
\begin{align*}
& \sum_{\mathbf{i}_{S}}\left\vert L\left( e_{i_{1}},...,e_{i_{k}},\mathbf{z}%
,...,\mathbf{z}\right) \right\vert ^{s_{k}} \\
& \leq \left( \left( \frac{2}{\sqrt{\pi }}\right) ^{\frac{2k\left(
k-1\right) }{p}}\cdot \left( B_{\mathbb{C},k}^{\mathrm{mult}}\right) ^{\frac{%
p-2k}{p}}\cdot \sup_{\mathbf{z}^{(1)},...,\mathbf{z}^{(k)}\in S_{\ell
_{p}^{n}}}\left\vert L\left( \mathbf{z}^{(1)},...,\mathbf{z}^{(k)},\mathbf{z}%
,...,\mathbf{z}\right) \right\vert \right) ^{s_{k}} \\
& \leq \left( \left( \frac{2}{\sqrt{\pi }}\right) ^{\frac{2k\left(
k-1\right) }{p}}\cdot \left( B_{\mathbb{C},k}^{\mathrm{mult}}\right) ^{\frac{%
p-2k}{p}}\cdot \frac{(m-k)!\cdot m^{m}}{(m-k)^{m-k}\cdot m!}\left\Vert
P\right\Vert \right) ^{s_{k}}\text{.}
\end{align*}%
Thus 
\begin{align*}
& \left( \sum_{|\alpha |=m}\left\vert a_{\alpha }\right\vert ^{\frac{2mp}{%
mp+p-2m}}\right) ^{\frac{mp+p-2m}{2mp}} \\
& \leq \left( \frac{m!}{(m-k)!}\right) ^{\frac{q-1}{q}}\cdot
K_{s_{k},p}^{m-k}\cdot \frac{(m-k)!\cdot m^{m}}{(m-k)^{m-k}\cdot m!}\cdot
\left( \frac{2}{\sqrt{\pi }}\right) ^{\frac{2k\left( k-1\right) }{p}}\cdot
\left( B_{\mathbb{C},k}^{\mathrm{mult}}\right) ^{\frac{p-2k}{p}}\Vert P\Vert
.
\end{align*}
\end{proof}

\ 

\end{document}